\documentclass[final]{dmtcs-episciences}

\usepackage[utf8]{inputenc}
\usepackage{subfigure}

\usepackage{amsfonts}
\usepackage{amssymb}
\usepackage{amsthm}
\usepackage{amsmath}

\newtheorem{lemma}{Lemma}
\newtheorem{theorem}{Theorem}

{\theoremstyle{definition}

\newtheorem{example}{Example}
}

\usepackage[round]{natbib}

\author{Paul Levrie\affiliationmark{1,2}
 \and John Campbell\affiliationmark{3}}
\title[Series acceleration formulas]{Series acceleration formulas obtained from\\ 
 experimentally discovered hypergeometric\\ 
 recursions} 
\affiliation{
 UAntwerpen, Belgium\\
 KU Leuven, Belgium\\
 York University, Canada}
\received{2022-05-14}
\revised{2022-11-08}
\accepted{2022-12-04}

\begin{document}
\publicationdetails{24}{2022}{2}{12}{9557}
\maketitle
\begin{abstract}
 In 2010, Kh.\ Hessami Pilehrood and T.\ Hessami Pilehrood introduced generating function identities used to obtain series accelerations for values of 
 Dirichlet's $\beta$ function, via the Markov--Wilf--Zeilberger method. Inspired by these past results, together with related results introduced by Chu et al., 
 we introduce a variety of hypergeometric recurrences. We prove these recurrences using the WZ method, and we apply these recurrences to obtain series 
 acceleration identities. We introduce a family of summations generalizing a Ramanujan-type series for $\frac{1}{\pi^2}$ due to Guillera, and a family of 
 summations generalizing an accelerated series for Catalan's constant due to Lupa\c{s}, and many related results. 
\end{abstract}

\section{Introduction}
 In a recent article by \cite{Campbell2022Discrete}, the hypergeometric transform 
\begin{equation}\label{fromDML}
 {}_{3}F_{2}\!\!\left[ \begin{matrix} a, b, 1 
 \vspace{1mm}\\ 
 c, 2 \end{matrix} \ \Bigg| \
 {1} \right] = - \frac{c-1}{(a-1) (b-1)} + \frac{c-1}{(a-1) (b-1)} 
 {}_{2}F_{1}\!\!\left[ \begin{matrix} a - 1, b - 1 
 \vspace{1mm}\\ 
 c - 1 \end{matrix} \ \Bigg| \
 {1} \right] 
\end{equation}
 was proved using the Wilf--Zeilberger (WZ) 
 method \cite{PetkovsekWilfZeilberger1996} and was applied to formulate a new proof of Euler's formula $\zeta(2) = 
 \pi^2/6$. The identity in \eqref{fromDML} was also used in relation to some problems on binomial-harmonic sums given by \cite{WangChu2022}. In this article, 
 we provide a number of hypergeometric identities resembling \eqref{fromDML} that we obtained experimentally. We prove them with the 
 WZ method \cite{PetkovsekWilfZeilberger1996}, and use them to formulate and apply series acceleration methods. 
 Many previously published closed forms for fast converging series by many different authors, 
 including \cite{Ramanujan1914}, 
 \cite{Guillera2003,Guillera2008}, 
 \cite{HessamiPilehroodHessamiPilehrood2008,HessamiPilehroodHessamiPilehrood2010}, 
 and \cite{Lupas2000}, 
 are special cases of our telescoping-based series acceleration techniques. 
 
 Our methods are related to the work of \cite{Catalan1867} and \cite{Fabry2020} on series evaluations, 
 and to the first author's previous work concerning series by 
 {W}allis, {F}orsyth, and {R}amanujan \cite{LevrieNimbran2018}; also, see \cite{Ojanguren2018}. Catalan and Fabry used a method due to \cite{Kummer1837} to accelerate the convergence of some series.
 The repeated application of Kummer's transformation leads to new series. One of the examples given by Fabry is based on the series 
 \begin{equation} \label{Fabry}
 \frac{\pi^2}{12} = \sum_{n=1}^\infty \frac{(-1)^{n-1}}{n^2} .
 \end{equation}
 This example was later taken over by \cite{Knopp1990} in his classic text, 
 where we find the following exercise \cite[129.c), p. 272]{Knopp1990}: Prove that 
 $$\sum_{n=1}^\infty \frac{(-1)^{n-1}}{n^2(n+1)^2\ldots(n+p-1)^2} 
 = \frac{5p+2}{4(p+1)p!^2}-\frac{p(p+1)^3}{4}\sum_{n=1}^\infty\frac{(-1)^{n-1}}{n^2(n+1)^2\ldots(n+p+1)^2}. $$ 
 Using this recurrence with $p=1$ transforms the series in \eqref{Fabry} into a new series that converges faster. Repeated applications 
 of the recurrence lead to the series 
 $$ \frac{\pi^2}{12} = \frac{7}{8} - \frac{8}{4} \cdot \left(\frac{17}{576}-\frac{192}{4} \cdot \left(\frac{27}{345600}-\frac{1080}{4} \cdot \left( \ldots \right)\right) \right) = 
 \frac{1}{4}\sum_{n=1}^\infty (-1)^{n-1}\frac{10n-3}{n^2(2n-1) \binom{2n}{n}}$$
 with convergence rate $\tfrac{1}{4}$ \cite[p.138]{Fabry2020}. The above recurrence can be rewritten and generalized in terms of hypergeometric functions. 
 This result is given in the second part of Theorem \ref{theoremfurther3F2n1} below. 

 While our hypergeometric recurrences given in Sections \ref{afterIntro}--\ref{section6F51} are of interest in their own right, the main results of this 
 article are our new telescoping proofs of known identities equating fundamental constants with fast convergent series, together with 
 related evaluations for new, fast converging series. We provide, for example, new proofs and generalizations of remarkable formulas for $\frac{1}{\pi^2}$ 
 that were discovered and proved by Guillera 
 \cite{Guillera2003,Guillera2008}
 and that are indicated in Table \ref{table:1}. It is remarkable how the WZ proof certificates for many of our recurrences are 
 hugely complicated compared to the elegance and simplicity of our closed forms obtained from these recurrences. However, the WZ proofs are 
 straightforward compared to Fabry/Kummer-like proofs of these results. These require some ingenuity and are all based on special identities enabling a 
 telescoping effect, as we explain in more detail in Section \ref{sectionnoncomp}. 

\subsection{Related literature}
 Much of this article is inspired by the methods and results introduced by \cite{HessamiPilehroodHessamiPilehrood2010}, 
 who introduced 
 a number of fast convergent series for Catalan's constant $G = \sum _{n=0}^{\infty } \frac{(-1)^n}{(2 n+1)^2}$. 
 These series are indicated in Table \ref{table:1}, which also includes a number of other previously published 
 series involved in or otherwise related to our work. 

\begin{table}[h!]
\centering
\begin{tabular}{|c c c|} 
 \hline
 Constant & Summand & Convergence rate \\ [0.5ex] 
% \hline\hline \\ [-2ex]
 $G$ & 
 $\frac{1}{8} \frac{2^{4n} (6n-1) }{ n^2 (2n-1) \binom{4n}{2n} \binom{2n}{n}} $ & 
 $\frac{1}{4}$ \\ [1.5ex] % \label{LupasG} 
 $G$ & $ - \frac{1}{64} \frac{ (-2^{8})^{n} (40 n^2 - 24 n + 3) }{ n^3 (2n-1)\binom{4n}{2n}^2 \binom{2n}{n} } $ 
 & $ \frac{1}{4} $ \\ [1.5ex]
 $G$ & $-\frac{1}{2}\frac{(-2^{3})^n (3 n - 
 1)}{ n^3 \binom{2 n}{n}^3 }$ & $ \frac{1}{8}$ \\ [1.5ex]
 $G$ & $\frac{1}{64} 
 \frac{2^{8n} (580 n^2 - 184 n + 15) }{n^3 (2n-1) \binom{6n}{3n} \binom{6n}{4n} \binom{4n}{2n}} $ & $\frac{4}{729}$ 
 \\ [1.5ex]
 $\pi^2$ & $ 2 \frac{ 2^{4 n} (3 n-1) }{n^3 \binom{2 n}{n}^3}$ 
 & $ \frac{1}{4} $ \\ [1.5ex]
 $\pi^2$ & $ 4 \frac{ 2^{4 n} (n+1) (3 n+1) }{n (2 n+1)^2 \binom{2
 n}{n}^3}$ & $\frac{1}{4}$ \\ [1.5ex]
 $\pi^2$ & $ 
 6 \frac{21 n-8}{n^3 \binom{2 n}{n}^3} 
 $ & $\frac{1}{64}$ \\ [2ex]
 $ \pi^{-1} $ & $ 
 \frac{1}{2} (-2^{-6})^n \binom{2n}{n}^3 (4n+1) $ & $1$ \\ [1ex]
 $ \pi^{-1} $ & $\frac{1}{16} \left( 2^{-12} \right)^{n} \binom{2 n}{n}^3
 (42 n+5) $ & $\frac{1}{64}$ \\ [1ex]
 $ \pi^{-2} $ & $\frac{1}{8} \left(- 2^{-12} \right)^n \binom{2 
 n}{n}^5 \left(20 n^2+8 n+1\right)$ & $\frac{1}{4}$ \\ [1ex]
 $\pi^{-2}$ & $\frac{1}{128} \left(-{2^{-20}}\right)^n \binom{2 n}{n}^5 
 \left(820 n^2+180 n+13\right) $ & $\frac{1}{1024}$ \\ [1ex] 
 \hline
\end{tabular}
\caption{Previously known series involved in or otherwise related to our work.
 For each of the series indicated in this table, we let it be understood that the lower limit is $n = 1$ 
 if the denominator of the corresponding summand vanishes for $n = 0$ 
 and that the lower limit is $n = 0$ otherwise. References for the above series are given, respectively, 
 as follows: \cite{HessamiPilehroodHessamiPilehrood2010}, \cite{Campbell2022AMEN,HessamiPilehroodHessamiPilehrood2010,Lima2012,Lupas2000}, \cite{Guillera2008,Lima2012}, 
 \cite{HessamiPilehroodHessamiPilehrood2010}, \cite{Guillera2008,HessamiPilehroodHessamiPilehrood2010}, 
 \cite{Campbell2022AMEN}, \cite{HessamiPilehroodHessamiPilehrood2008}, \cite{Bauer1859}, \cite{Ramanujan1914}, \cite{Guillera2003}, 
 and \cite{Guillera2008}. }
\label{table:1}
\end{table}

 The second formula for $G$ shown in Table \ref{table:1} is used in the Wolfram Language to calculate $G$ \cite{MarichevSondowWeisstein}. This 
 motivates the techniques and applications given in this article, in which a new acceleration method is applied to generalize this formula for $G$ and many 
 other fast convergent series formulas. 

 Our series recursions and accelerations are also related to the work by 
 \cite{Chu2006} and \cite{ChuZhang2014}. In \cite{ChuZhang2014}, 
 the series rearrangement referred to as the modified Abel lemma on summation by parts was employed to obtain hypergeometric series recursions related to 
 our work. For a hypergeometric series $h(x, y)$ involving the variables $x$ and $y$ in its parameters, we prove recursions such as 
\begin{equation}\label{formofrec}
 h(x, y) = r_{1}(x, y) + r_{2}(x, y) \, h(x, y+1) 
\end{equation}
 for rational functions $r_{1}(x, y)$ and $r_{2}(x, y)$, and we then apply such recursions iteratively, yielding 
 accelerated series formulas. In contrast to \cite{ChuZhang2014}, we do not make use of any Abel-type summation results. Instead, we employ the WZ 
 method. Furthermore, our experimental approaches allow us to obtain recursions for ${}_{p}F_{q}(-1)$-families. It seems that ${}_{p}F_{q}(-1)$-series 
 are not considered in past references such as \cite{ChuZhang2014}. Also, 
 an advantage of our WZ-based methodologies is due to how the WZ difference 
 equation $$ F(n+1, k) - F(n, k) = G(n, k+1) - G(n, k) $$ may be manipulated in order to obtain new and nontrivial hypergeometric identities, using the WZ 
 pairs corresponding to our WZ proofs, together with WZ-based series manipulations as in 
 \cite{Campbell2021}. It appears that the Maple implementation of the WZ method
 cannot be applied to the recursions due to Chu and Zhang in \cite{ChuZhang2014}.
 This shows how the methods in \cite{ChuZhang2014} are nontrivially different from 
 ours, which rely on the WZ method. 

\subsection{Organization of the article}
 The rest of our article is mainly structured and organized in the manner indicated as follows. 
 In each of the following main sections, apart from Section \ref{sectionnoncomp} 
 and the Conclusion, we introduce recursions for families of ${}_{p}F_{q}(x)$-series for fixed $p$, $q$, and $x$. 
 The recurrences are 
 given in terms of the parameters of the generalized hypergeometric functions. We present WZ proofs of these 
 recursions, and then apply these recursions to obtain series acceleration identities, to obtain fast convergent series formulas for fundamental constants. 
 Given one such fast convergent series, 
 our method naturally gives rise to an infinite family, by altering the input parameters. 
 In this regard, we choose to emphasize, as in Sections \ref{subsectionfamilyG} and \ref{GeneralizationsGuillera} below, 
 our infinite families of generalizations of especially remarkable results due to \cite{Lupas2000} 
 and \cite{Guillera2003}. 

 In Section \ref{afterIntro}, we introduce experimentally discovered ${}_{3}F_{2}(1)$-recurrences, 
 yielding formulas for very fast convergent series, such as 
\begin{equation}\label{especially}
 \zeta(2) = \frac{7}{4} - 2\sum_{n=2}^\infty \frac{(3n-1)(7n-5)}{n^2(n-1)^2\binom{2n}{n}^3}, 
\end{equation}
 a series with convergence rate $\frac{1}{64}$. 

 In Section \ref{Sectionfirst3F2n1}, 
 we introduce experimentally discovered ${}_{3}F_{2}(-1)$-recurrences, 
 yielding an infinite family of series generalizing the algorithm used to compute $G$ in the Wolfram Language \cite{MarichevSondowWeisstein}. 

 In Section \ref{section5F4n1}, we introduce an experimentally discovered ${}_{5}F_{4}(-1)$-recurrence, 
 yielding new WZ proofs of two of Ramanujan's series for $\frac{1}{\pi}$, 
 along with the new formula 
\begin{equation}\label{Catalanmotivating}
 G = \frac{1}{16}\sum_{n=1}^{\infty} \frac{2^{6n}(336n^4 - 464n^3 + 216n^2 - 42n + 3)}{ n^3 (2n-1)^3 \binom{4n}{2n}^3}, 
\end{equation}
 which has a convergence rate of $\frac{1}{64}$. 

 In Section \ref{section6F51}, we introduce experimentally discovered ${}_{6}F_{5}(1)$-recurrences, 
 which we apply to obtain a variety of results, including new proofs/generalizations of 
 Guillera's series for $\frac{1}{\pi^2}$. 

 In Section \ref{sectionnoncomp}, we briefly describe 
 the experimental mathematics-based nature about how we had discovered the hypergeometric recurrences proved
 via the WZ method in this article. 

 In Section \ref{sectionConclusion}, we conclude by briefly describing 
 some further areas of research based on the techniques given in this article. 

\section{Experimentally discovered ${}_{3}F_{2}(1)$-recurrences}\label{afterIntro}
 For the sake of brevity, we assume familiarity with WZ theory \cite{PetkovsekWilfZeilberger1996}, generalized hypergeometric series, 
 the Pochhammer symbol, etc. Our first main result is highlighted as Theorem \ref{theoremfirst3F21} below. 

 A number of the hypergeometric recurrences introduced in this article are of the form indicated in 
 \eqref{formofrec}. As below, we are letting the summand of a hypergeometric series $h(x, y)$ as in \eqref{formofrec} 
 be denoted as $\text{summand}(x, y, k)$. 
 This notation will be useful in our generalizing a fast converging series formula 
 for $\frac{1}{\pi^2}$ due to Guillera, as in Section \ref{GeneralizationsGuillera}. 

\begin{lemma}\label{generalizeGuillera}
 If the recurrence in \eqref{formofrec} holds, then the infinite series 
\begin{equation}\label{gGsummand}
 \sum _{k=0}^{\infty} \text{\emph{summand}}(x,y,k) 
\end{equation}
 may be expressed as $$ \sum _{j=0}^{m - 1} \left(\prod _{i=0}^{j-1} r_{2}(x,y+i)\right) r_{1}(x,y+j)+\left(\prod _{i = 
 0}^{m-1} r_{2}(x,y+i)\right) \sum _{k=0}^{\infty} \text{\emph{summand}}(x,y+m,k) $$ for natural numbers $m$, 
 under the assumption that the above series converge, for given values of $x$ and $y$. 
\end{lemma}

\begin{proof}
 This follows in a direct way from repeated applications of \eqref{formofrec}. 
\end{proof}

 Ideally, we would want that 
\begin{equation*}
 \left(\prod _{i=0}^{m-1} r_{2}(x,y + i)\right) \sum _{k=0}^{\infty} \text{summand}(x,y+m,k) \to 0 
\end{equation*}
 as $m \to \infty$.  Although the above vanishing condition does not necessarily imply acceleration,  we would, ideally, want, from Lemma 
 \ref{generalizeGuillera},  to obtain a series $$ \sum _{j=0}^{\infty} \left(\prod _{i=0}^{j-1} r_{2}(x,y+i)\right) r_{1}(x,y+j) $$ that is new, convergent, 
 and faster convergent relative to \eqref{gGsummand}. 

 By symmetry, given  a recursion such as $h'(x, y) = r_{1}'(x, y) + r_{2}'(x, y) h'(x+1, y)$ 
 for a hypergeometric series $h'$, we may again apply Lemma \ref{generalizeGuillera}, 
 by setting $h(x, y)$ as $h'(y, x)$. 
 We may similarly generalize Lemma \ref{generalizeGuillera} 
 using recurrences of the form 
 $ h''(x, y) = r_{1}''(x, y) + r_{2}''(x, y) \, h''(x, y + \mu) $ for fixed $\mu \in \mathbb{N}$. 
 For the time being, we proceed by providing and proving a new ${}_{3}F_{2}(1)$-recursion of 
 the form indicated in \eqref{formofrec}. 

\begin{theorem}\label{theoremfirst3F21}
 The ${}_{3}F_{2}(1)$-transform 
\begin{align*}
 & {}_{3}F_{2}\!\!\left[ \begin{matrix}x,x+1,1 \vspace{1mm}\\ x+y,x+y+1 \end{matrix} \ \Bigg| \
 {1} \right] = \frac{2x+3y-1}{2(2y-1)}\\ 
 & \hspace{5cm} + \frac{(y-1)y(y+1)}{2(2y-1)(x+y)(x+y+1)}\cdot 
 {}_{3}F_{2}\!\!\left[ \begin{matrix}x,x+1,1 \vspace{1mm}\\ x+y+1,x+y+2 \end{matrix} \ \Bigg| \
 {1} \right]
\end{align*}
 holds true for $x$ and $y$ such that both sides of the above equality converge. 
\end{theorem}

\begin{proof}
 We set $x = -n$. So, it remains to prove that 
\begin{align*}
 & \sum_{k = 0}^{\infty} \frac{(1-n)_k (-n)_k}{(y-n)_k (-n+y+1)_k} \cdot \\ 
 & \frac{k^2 (2-4 y)+k \left(n (8 y-4)-8 y^2+2\right)+n^2 (2-4 y)+n \left(8 y^2-2\right)-3 y^3-2 y^2+y}{(2 n-3 y+1) (k-n+y) (k-n+y+1)} 
 \end{align*}
  constantly equals $1$.   Now, we restrict $n$ to being a member of the set $\mathbb{N}_{0}$.  
 In this case, Pochhammer symbols of the form $(-n)_{k}$ vanish for $k > n$, with $k \in \mathbb{N}_{0}$. 
 So, for $n \in \mathbb{N}_{0}$, we may replace the upper limit of the above series with $n$. 
 Now, we may apply the WZ method, giving us a WZ proof certificate
 that is given explicitly in the 
 first draft available via the \url{https://hal.archives-ouvertes.fr/hal-03659300/} site. 
 So, using the WZ method, we have shown that the above infinite series identity holds for $n \in \mathbb{N}_{0}$. We may thus apply 
 Carlson's theorem \cite[p.\ 39]{Bailey1935} 
 as in \cite{EkhadZeilberger1994} 
 to show that the same series identity holds for real $n$. 
\end{proof}

\begin{theorem}\label{notlemma}
 The ${}_{3}F_{2}(1)$-transform 
\[
{}_{3}F_{2}\!\!\left[ \begin{matrix}x,x+1,1 \vspace{1mm}\\ x+y,x+y+1 \end{matrix} \ \Bigg| \
 {1} \right] = 1 + \frac{x(x+1)}{(x+y)(x+y+1)}\cdot {}_{3}F_{2}\!\!\left[ \begin{matrix}x+1,x+2,1 \vspace{1mm}\\ x+y+1,x+y+2 \end{matrix} \ \Bigg| \
 {1} \right]
\]
 holds true for $x$ and $y$ such that both sides of the above equality converge. 
\end{theorem}

\begin{proof}
  Again, we set $x = -n$. So, it remains to prove that   $$ \sum_{k = 0}^{\infty}   \frac{(1-n)_k (-n)_k}{(y-n)_k (-n+y+1)_k}   \frac{y (2 k-2 n+y+1)}{(k 
  -n+y) (k-n+y+1)} = 1. $$   Again, we consider the case for $n \in \mathbb{N}_{0}$,  so that we may again replace the upper parameter with $n$.    
 Again, we may apply the WZ method.  The corresponding WZ proof certificate is given below:  
 $$ \frac{k (k-n+y) (k-n+y+1) \left(-2 k n+k y-k+2 n^2-2 n y+2 n-y+1\right)}{(k-n-1) (k-n) (n-y) (n-y+1) (2 k - 
 2 n+y+1)}. $$
 We may thus apply Carlson's theorem, as in our proof of Theorem \ref{theoremfirst3F21}. 
\end{proof}

\subsection{A series acceleration method}\label{subsseriesaccel}
 Let us define $$ s(x,y)=\frac{x}{x+y}\cdot {}_{3}F_{2}\!\!\left[ \begin{matrix}x,x+1,1 
 \vspace{1mm}\\ x+y,x+y+1 \end{matrix} \ \Bigg| \
 {1} \right] = \sum_{n=0}^\infty\frac{(x)_n (x)_{n+1}}{(x+y)_n (x+y)_{n + 
 1}} \ . $$ We rewrite the recurrences given as Theorems \ref{theoremfirst3F21} and \ref{notlemma}, respectively, so that 
\begin{equation}
s(x,y) = \frac{x(2x+3y-1)}{2(2y-1)(x+y)} + \frac{(y-1)y(y+1)}{2(2y-1)(x+y)^2}s(x,y+1)
\label{eq:1}
\end{equation}
and that 
\begin{equation}\label{eq:2}
s(x,y) = \frac{x}{x+y} + \frac{x^2}{(x+y)^2}s(x+1,y) .
\end{equation}
Note that we know the value of $s(x,y)$ in some cases: 
\begin{equation}\label{svalues}
 s(1,2)= 7 - \frac{2}{3}\pi^2, s(\tfrac{1}{2},2)= 3 - \frac{9}{32}\pi^2, s(1,\tfrac{3}{2})=10 - 3\pi, s(\tfrac{1}{2},\tfrac{3}{2})=2 - \frac{16}{3\pi}. 
\end{equation}
 Indeed, we have, for instance, that: $$ s(1,2) = \sum_{n=0}^\infty \frac{4}{(n + 3) (n + 2)^2 (n + 1)} = \sum_{n=0}^\infty \left(\frac{2}{n + 
 1}-\frac{2}{n+3} -\frac{4}{(n + 2)^2}\right), $$ leading to the given closed form for 
 $s(1, 2)$. 
 We can deal with the second one in a similar manner. For the third and 
 fourth ones, 
 we can use a closed form for the partial sum of the series, as in: $$ s(\tfrac{1}{2},\tfrac{3}{2}) = \sum_{n = 1}^\infty \frac{2}{(n+1)(2 n - 
 1)}\frac{\binom{2n}{n}^2}{2^{4n}} \ \ \ \mbox{with} \ \ \ \sum_{n=1}^{k-1} \frac{2}{(n+1)(2n-1)}\frac{\binom{2n}{n}^2}{2^{4n}} = 2 - 
 \frac{8}{3} \frac{k(4k-1)}{2k-1}\frac{\binom{2k}{k}^2}{2^{4k}}. $$ We then take the limit and use Wallis's product formula for $\pi$. 

 We can use the $s$-recurrences above to accelerate the convergence of the corresponding series. Iterating the first recurrence leads to the following 
 accelerated series, via Lemma \ref{generalizeGuillera}: 
 \[ s(x,y)=\sum_{n=0}^\infty \frac{x(2x+3(y+n)-1)}{2(2(y+n)-1)(x+y+n)} \frac{(y-1)_n (y)_n (y+1)_n}{2^{2n}(y-\frac{1}{2})_n (x 
 + y)^2_n} . \] 
 For the four special values given above, we find the following four series: 
\begin{align}
 \zeta(2) & = \frac{7}{4} - \sum_{n=2}^\infty \frac{3n+1}{(n+1)n^2(n-1)\binom{2n}{n}}, \nonumber \\ 
 \zeta(2) & = \frac{16}{9} - \sum_{n=2}^\infty \frac{2^{4n}}{(2n+1)n(n-1)\binom{2n}{n}^3}, \label{Basel16} \\ 
\pi & = -2\sum_{n=0}^\infty \left( 2^{-4} \right)^{n}\binom{2n}{n}\frac{6n+5}{(2n+3)(2n+1)(2n-1)}, \nonumber \\ 
\frac{1}{\pi} & = - 12 \sum_{n=0}^\infty \left( 2^{-8} \right)^{n}\binom{2n}{n}^3\frac{n^2}{(2n-1)(2n-3)}. \nonumber 
\end{align}
 The fast convergent series evaluation in \eqref{Basel16} is closely related to the series for $\pi^2$ from \cite{Campbell2022AMEN} shown in Table 
 \ref{table:1}, noting the convergence rate of $\frac{1}{4}$ in both cases. 

 If we use \eqref{eq:2} in \eqref{eq:1}, we get a recurrence relating $s(x,y)$ to $s(x+1,y+1)$, and other accelerated sums, including \eqref{especially}. 
 This series is especially remarkable for many reasons, perhaps most notably due to its very fast convergence rate of $\frac{1}{64}$. In the same manner 
 that we have obtained \eqref{especially}, the following formulas may be proved: 
\begin{align*}
 \zeta(2) & = \frac{16}{9} - 2\sum_{n=1}^\infty \frac{2^{4n}(14n^2+11n + 
 1)(n+1)}{n (4n+3)(4n+1)^2(2n-1)\binom{2n}{n}\binom{4n}{2n}^2}, \\ 
 \pi & = \frac{10}{3} -2\sum_{n=1}^\infty \frac{(14n+11)(2n + 
 1)(3n+1)\binom{2n}{n}}{(4n+3)(4n+1)^2(2n-1)\binom{4n}{2n}^2}, \\ 
 \frac{1}{\pi} 
 & = - \frac{1}{8} \sum_{n=0}^\infty\left( 2^{-12} \right)^{n}\binom{2n}{n}^3 \frac{(6 n + 
 1)(14n-3)(2n+1)}{(2n-1)^2}. 
\end{align*}
 This leads to an infinite number of possibilities. 

\section{Experimentally discovered ${}_{3}F_{2}(-1)$-recurrences}\label{Sectionfirst3F2n1}
 In view of the closed forms shown in \eqref{svalues}, it is not clear as to how to obtain accelerated series for Catalan's constant $G = \sum _{n = 
 0}^{\infty } \frac{(-1)^n}{(2 n+1)^2}$. This motivates the exploration of alternating variants and generalizations Theorems \ref{theoremfirst3F21} 
 and \ref{notlemma}. Defining $t(x, y)$ so that 
\begin{equation}\label{txydefinition}
 t(x,y) = \sum_{n=0}^\infty \frac{(x)_n (x)_{n+1}}{(x+y)_n (x+y)_{n+1}} (-1)^n, 
\end{equation}
 we find that $t\big(\frac{1}{2}, 2\big) = \frac{3}{8} (6 G-5)$. Note that
 $$ t(x,y)=\frac{x}{x+y}\cdot {}_{3}F_{2}\!\!\left[ \begin{matrix}x,x+1,1 \vspace{1mm}\\ x+y,x+y+1 \end{matrix} \ \Bigg| \
 {-1} \right]. $$ 
 This motivates the use of analogues of our above acceleration methods, in order to obtain accelerated series for $G$. 

 The recurrence highlighted below is remarkable in the sense that it is far from clear as to how our proofs for the preceding Theorems 
 could be mimicked or reformulated so as to prove the identity highlighted as Theorem \ref{againnotlemma}. 

\begin{theorem}\label{againnotlemma}
 The ${}_{3}F_{2}(-1)$-recurrence 
\begin{equation}
 t(x,y) = \frac{x(2x^2+6xy+5y^2+2x+4y)}{4(x+y)^2(x+y+1)} - \frac{y^2(y+1)(y+2)}{4(x+y)^2(x+y+1)^2}t(x,y+2)
 \label{eq:1b}
\end{equation}
 holds true for suitably bounded $x$ and $y$. 
\end{theorem}

\begin{proof}
 Rewrite $x$ as $-n$. So, it remains to show that $$ \sum_{k=0}^{\infty} \frac{(-1)^k (-n)_k (-n)_{k + 1}}{(-n+y)_k (-n+y)_{1+k}} \left( 
 -\frac{(n-y)^2 (-n+y+1) \left(\frac{(y+1) (y+2) y^2}{(k-n+y) (k-n+y+1)^2 (k-n+y+2)}+4\right)}{n \left(-6 n y+2 (n-1) n+5 y^2+4 y\right)} \right)$$ 
 constantly equals $1$. In line with our previous WZ proofs, we temporarily restrict our attention to the case whereby $n \in \mathbb{N}_{0}$. 
 Again, the vanishing of $(-n)_{k}$ for $k > n \in \mathbb{N}_{0}$ produces a finite sum. 
 A WZ proof certificate for this finite sum is given in the first article draft available via the HAL site previously indicated. 
 Once again, we may apply Carlson's theorem. 
\end{proof}

 It is common to denote quotients of Pochhammer symbols so that 
 $$ \!\!\left[ \begin{matrix} \alpha_{1}, \ldots, \alpha_{p} \vspace{1mm} \\ 
 \beta_{1}, \ldots, \beta_{q} \end{matrix} \right]_{n} = 
 \frac{ \left( \alpha_{1} \right)_{n} \cdots \left( \alpha_{p} \right)_{n} }{ 
 \left( \beta_{1} \right)_{n} \cdots \left( \beta_{q} \right)_{n} }, $$ 
 as in the series acceleration identity shown below. 

\begin{theorem}\label{theoremtxy}
 The expression $t(x, y)$ equals the accelerated infinite series 
\begin{align*}
 \sum_{j=0}^{\infty} 
 & \left( -2^{-2} \right)^{j} 
 \frac{x y^2 (y+1) (y+2) \left(6 x (2 j+y)+5 (2 j+y)^2+4 (2 j+y)+2 x^2+2 x\right)}{4 (x + 
 y)^2 (x+y+1)^2 (2 j+x+y)^2 (2 j+x+y+1)} \cdot \\ 
 & \!\!\left[ \begin{matrix} \frac{y}{2}+1, \frac{y}{2}+1, \frac{y}{2}+2, \frac{y+3}{2} \vspace{1mm} \\ 
 \frac{x+y}{2} + 1, \frac{x+y}{2} + 1, 
 \frac{x+y+3}{2}, \frac{x+y+3}{2} \end{matrix} \right]_{j-1} 
 \end{align*}
 if both sides converge for $x$ and $y$. 
\end{theorem}

\begin{proof}
 Through repeated applications of the recurrence given as Theorem \ref{againnotlemma}, we may rewrite $t(x, y)$ as 
\begin{align*}
 & \sum _{j=0}^{m-1} \frac{(-1)^j \left(\prod _{i=0}^{j-1} \frac{(2 i+y)^2 (1+2 i+y) (2+2 i+y)}{4 (2 i+x+y)^2 (1+2 i+x+y)^2}\right) 
x \left(2 x+2
 x^2+4 (2 j+y)+6 x (2 j+y)+5 (2 j+y)^2\right)}{4 (2 j+x+y)^2 (1+2 j+x+y)} \\ 
 & + (-1)^m \left(\prod _{i=0}^{m-1} \frac{(2 i+y)^2 (1+2 i+y) (2+2 i+y)}{4 (2 i+x+y)^2 (2 i+x + y + 1)^2}\right) t(x,y+2 m). 
\end{align*}
 Taking $m \to \infty$, it is easily seen that the latter term vanishes. 
\end{proof}

\begin{example}
 Setting $x = \frac{1}{2}$ and $y = 1$ in Theorem \ref{theoremtxy}, we obtain the accelerated series identity $$ 8-4 \pi = \sum _{j=1}^{\infty} 
 \frac{ \left(-2^{4} \right)^j \left(40 j^2-12 j-1\right)\binom{2 j}{j} }{ j (2 j-1)^2 (4 j+1)\binom{4 j}{2 j}^2 }. $$ The above series is of 
 convergence rate $\frac{1}{4}$. 
\end{example}

\subsection{A family of accelerated series for $G$}\label{subsectionfamilyG}
 By setting $x =- \frac{1}{2}$ and $y = 2$ in Theorem \ref{theoremtxy}, 
 or by setting $x = \frac{1}{2}$ and $y = 2$ in Theorem \ref{theoremtxy}, 
 in either case, this gives us 
 an equivalent 
 version of the series used to calculate $G$ in the Wolfram Language \cite{MarichevSondowWeisstein}.
 For example, in the latter case, we obtain that 
\begin{equation}\label{alternateLupasWolfram}
 \frac{40}{3}-16 G 
 = \sum _{j=1}^{\infty} \frac{ \left(-2^{8}\right)^j \left(40j^2+28 j+3\right)}{j^2 (4 j+1)^2 (4 j+3) \binom{2 j}{j} \binom{4 j}{2 j}^2}, 
\end{equation}
 and adding 16 times the series given by the second formula for $G$ in Table \ref{table:1} gives us a series with partial sums that are evaluable in 
 closed form. 

\begin{example}
 Set $x = \frac{3}{2}$ and $y = 2$ in Theorem \ref{theoremtxy}. This gives us 
 $$ \frac{16 G}{3}-\frac{664}{135} 
 = \sum_{j=1}^{\infty} \frac{ \left(-2^{8}\right)^j \left(40 j^2+52 j+15\right)}{j^2 (4 j+1)^2 (4 j+3)^2 (4 j+5) \binom{2 j}{j} \binom{4 j}{2 j}^2}. $$
\end{example}

\begin{example}
 Set $x = \frac{5}{2}$ and $y = 2$ in Theorem \ref{theoremtxy}. This gives us that 
 $$ \frac{23048}{70875}-\frac{16 G}{45}
 = \sum_{j = 1}^{\infty} 
 \frac{ \left( -2^{8} \right)^j \left(40 j^2+76 j+35\right)}{j^2 (4 j+1)^2 (4 j+3)^2 (4 j+5)^2 (4 j+7) \binom{2 j}{j} \binom{4 j}{2 j}^2}. $$ 
\end{example}
 
 Continuing in this manner gives us an infinite family of extensions of \eqref{alternateLupasWolfram}. It appears that our techniques for generating 
 such families of accelerated series have not been considered previously, as in past literature related 
 to \cite{HessamiPilehroodHessamiPilehrood2010}, including the work by 
 \cite{AdellLekuona2016,AdellLekuona2015,HessamiPilehroodHessamiPilehrood2011,HessamiPilehroodHessamiPilehroodTauraso2012,HessamiPilehroodHessamiPilehroodTauraso2014,WangXu2021}. 

\subsection{Further ${}_{3}F_{2}(-1)$-recurrences}

\begin{theorem}\label{theoremfurther3F2n1}
 The ${}_{3}F_{2}(-1)$-recurrences 
\begin{align*} 
 & {}_{3}F_{2}\!\!\left[ \begin{matrix}x,x,1 \vspace{1mm}\\ x+y,x+y \end{matrix} \ \Bigg| \
 {-1} \right] = 
 1- \frac{x^2}{(x + y)^2} \cdot {}_{3}F_{2}\!\!\left[ \begin{matrix}x+1,x+1,1 \vspace{1mm}\\ x+1+y,x+1 + y \end{matrix} \ \Bigg| \
 {-1} \right] \\ 
 & {}_{3}F_{2}\!\!\left[ \begin{matrix}x,x,1 \vspace{1mm}\\ x+y,x+y \end{matrix} \ \Bigg| \
 {-1} \right] = \frac{2x^2 + 6xy + 5y^2 + y}{4(x + y)^2} \\ 
 & \hspace{5cm}-\frac{y(y + 1)^3}{4(x + y)^2(x + y + 1)^2}\cdot {}_{3}F_{2}\!\!\left[ \begin{matrix}x,x,1 \vspace{1mm}\\ x+y+2,x+y+2 \end{matrix} \ \Bigg| \
 {-1} \right]
\end{align*}
 hold true, if all of the above series converge. 
\end{theorem}

\begin{proof}
 Write $x = -n$. We may mimic our previous WZ proofs, via Carlson's theorem, to prove the first out of the two ${}_{3}F_{2}(-1)$-recurrences shown 
 above. A corresponding WZ certificate is shown below: $$ \frac{k y (k-n+y)^2 \left(-2 k n+k y-2 k+2 n^2-2 n y+4 n-2 y+2\right)}{(k-n-1)^2 (n - y + 
 1)^2 \left(2 k^2-4 k n+2 k y+2 n^2-2 n y+y^2\right)}. $$ We may apply the same approach to prove the latter ${}_{3}F_{2}(-1)$-formula under 
 consideration. A WZ proof certificate is given explicitly in the previously provided HAL URL. 
\end{proof}

 With the notation $$ s(x,y) = {}_{3}F_{2}\!\!\left[ \begin{matrix}x,x,1 \vspace{1mm}\\ x+y,x+y \end{matrix} \ \Bigg| \
 {-1} \right], $$ we may express the recurrences in Theorem \ref{theoremfurther3F2n1} 
 as follows: 
\begin{align*}
s(x,y) & = 1- \frac{x^2}{(x + y)^2} s(x+1,y), \\ 
s(x,y) & = \frac{2x^2 + 6xy + 5y^2 + y}{4(x + y)^2}- \frac{y(y + 1)^3}{4(x + y)^2(x + y + 1)^2} s(x,y+2). 
\end{align*}
 If we use the second recurrence to calculate $s(\tfrac{1}{2},2) = \frac{9}{2}\,(1-\frac{\pi}{4})$, $s(1,2) = 4\,(3-4\ln (2))$, 
 and $s(\tfrac{1}{2},1) = G$, we 
 find the corresponding series in Lupa\c{s} \cite[p. 6]{Lupas2000}. If we combine the recurrences to write $s(x,y)$ as a function of $s(x+1,y+2)$ and of 
 $s(x+2,y+2)$, we obtain, for $x=\frac{1}{2}$ and $y=1$, a series for Catalan's constant proved in \cite{HessamiPilehroodHessamiPilehrood2010}. 

\section{Experimentally discovered ${}_{5}F_{4}(-1)$-recurrences}\label{section5F4n1}

\begin{theorem}
 The recurrences
\begin{align*}
 & \left(x+\frac{y-1}{2}\right)\cdot {}_{5}F_{4}\!\!\left[ \begin{matrix} x,x,x,x + 
 \frac{y+1}{2},1 \vspace{1mm}\\ x+y,x+y,x+y,x+\frac{y-1}{2} \end{matrix} \ \Bigg| \ {-1} \right] = x+\frac{y-1}{2} \\ 
 & - \frac{x^3}{(x+y)^3} \left(x+1+\frac{y-1}{2}\right)\cdot 
 {}_{5}F_{4}\!\!\left[ \begin{matrix} x+1,x+1,x+1,x+1+\frac{y+1}{2},1 \vspace{1mm}\\ x+1+y,x+1+y,x+1+y,x+1+\frac{y-1}{2} \end{matrix} 
 \ \Bigg| \ {-1} \right] \\ 
 & \left(x+\frac{y-1}{2}\right)\cdot {}_{5}F_{4}\!\!\left[ \begin{matrix} x,x,x,x+\frac{y+1}{2},1 \vspace{1mm}\\ x+y,x+y,x+y,x+\frac{y-1}{2} \end{matrix} \ \Bigg| \ {-1} \right] = \frac{x-1}{2}+y \\ 
 & - \frac{y^3}{(x+y)^3} \left(x+\frac{y}{2}\right)\cdot
 {}_{5}F_{4}\!\!\left[ \begin{matrix} x,x,x,x+\frac{y+2}{2},1 \vspace{1mm}\\ x+y+1,x+y+1,x+y+1,x+\frac{y}{2} \end{matrix} \ \Bigg| \ {-1} \right]
\end{align*}
 hold true, if all of the above series converge. 
\end{theorem}

\begin{proof}
 Again, we set $x = -n$, and we may apply the WZ method and Carlson's theorem in much the same way as before. WZ proof certificates for each out of the 
 two recursions shown above are given explicitly in via the previously provided HAL site. 
\end{proof}

 With the notation: $$ s(x,y)=	\left(x+\frac{y-1}{2}\right)\cdot {}_{5}F_{4}\!\!\left[ \begin{matrix} x,x,x,x+\frac{y+1}{2},1 \vspace{1mm}\\ x+y,x+y,x+y,x+\frac{y-1}{2} \end{matrix} \ \Bigg| \ {-1} \right], $$ 
 we find that: 
\begin{align*}
s(x,y) & = x+\frac{y-1}{2}- \frac{x^3}{(x+y)^3} s(x+1,y), \\ 
 s(x,y) & = \frac{x-1}{2}+y- \frac{y^3}{(x+y)^3} s(x,y+1). 
\end{align*}
 Note that for $x=\tfrac{1}{2}, y=\tfrac{1}{2}$, $s(x,y)$ is the Bauer--Ramanujan series for $\frac{1}{\pi}$ of convergence rate $1$ shown in Table 
 \ref{table:1}. If we try to accelerate this series with one of the recurrences, we get the Bauer--Ramanujan series again in both cases. Hence no 
 acceleration. However, if we use them one after the other, then we obtain Ramanujan's series for $\frac{1}{\pi}$ of convergence rate $\frac{1}{64}$, 
 which is indicated in Table \ref{table:1}. For $x=\tfrac{1}{2}, y=1$, we have that $$s(x,y)= \frac{1}{2}\sum_{n=0}^{\infty} (-1)^n \frac{1}{(2n+1)^2} = 
 \frac{G}{2}.$$ Using both recurrences one after the other leads to \eqref{Catalanmotivating}. For $x=1, y=1$, we obtain, in this case, the series for 
 $ \pi^2$ of convergence rate $\frac{1}{64}$ indicated in Table \ref{table:1}. 

\section{Experimentally discovered ${}_{6}F_{5}(1)$-recurrences}\label{section6F51}

\begin{theorem}\label{experimental6F51}
 The recurrences
 \begin{align*}
 & \left(x+\frac{y-1}{2}\right)\cdot {}_{6}F_{5}\!\!\left[ \begin{matrix} x,x,x,x,x+\frac{y+1}{2},1 \vspace{1mm}\\ x+y,x+y,x+y,x+y,x+\frac{y-1}{2} \end{matrix} \ \Bigg| \ {1} \right] = x+\frac{y-1}{2} \\ 
 & + \frac{x^4}{(x + y)^4} \left(x+1+\frac{y-1}{2}\right)\cdot 
 {}_{6}F_{5}\!\!\left[ \begin{matrix} x+1,x+1,x+1,x+1,x+1+\frac{y+1}{2},1 \vspace{1mm}\\ x+1+y,x+1+y,x+1+y,x+1+y,x+1+\frac{y-1}{2} \end{matrix} \ \Bigg| \ {1} \right] \\ 
 & \left(x+\frac{y-1}{2}\right)\cdot {}_{6}F_{5}\!\!\left[ \begin{matrix} x,x,x,x,x+\frac{y+1}{2},1 \vspace{1mm}\\ x+y,x+y,x+y,x + 
 y,x+\frac{y-1}{2} \end{matrix} \ \Bigg| \ {1} \right] = \frac{5y^2 + 6(x - 1)y + 2(x - 1)^2}{4(2y - 1)} \\ 
 & - \frac{y^5}{2(2y - 1)(x + y)^4} \left(x+\frac{y}{2}\right)\cdot 
 {}_{6}F_{5}\!\!\left[ \begin{matrix} x,x,x,x,x+\frac{y+2}{2},1 \vspace{1mm}\\ x+y+1,x+y+1,x+y+1,x+y+1,x+\frac{y}{2} \end{matrix} \ \Bigg| \ {1} \right]
\end{align*}
 hold true, if all of the above series converge. 
\end{theorem}

\begin{proof}
 Again, we write $x = -n$, and we may again mimic our WZ-based proofs given as above. WZ proof certificates for each case are given explicitly in the HAL 
 site previously given. 
\end{proof}

\subsection{Generalizations of Guillera's series}\label{GeneralizationsGuillera}
 With the notation $$ s(x,y) = \left(x+\frac{y-1}{2}\right)\cdot {}_{5}F_{4}\!\!\left[ \begin{matrix} x,x,x,x,x+\frac{y+1}{2},1 \vspace{1mm}\\ x+y,x+y,x+y,x+y,x+\frac{y-1}{2} \end{matrix} \ \Bigg| \ {1} \right],$$ 
 Theorem \ref{experimental6F51} gives us that 
\begin{align*}
 s(x,y) & = x+\frac{y-1}{2}+ \frac{x^4}{(x+y)^4} s(x+1,y), \\ 
 s(x,y) & = \frac{5y^2 + 6(x - 1)y + 2(x - 1)^2}{4(2y - 1)}- \frac{y^5}{2(2y - 1)(x + y)^4} s(x,y+1). 
\end{align*}
 Note that for $x = \tfrac{1}{2}, y=\tfrac{3}{2}$, $s(x,y)$ can be calculated from its partial sums:
 $$s(\tfrac{1}{2},\tfrac{3}{2})=\frac{1}{4} \sum_{n = 
 0}^{\infty} \left( 2^{-8} \right)^{n} \binom{2n}{n}^4 \frac{4n+3}{(n+1)^4}=4-\frac{32}{\pi^2} $$ If we try to accelerate this series with the first of 
 the recurrences, we get the same series. With the second one, we get Guillera's formula 
 for $\frac{1}{\pi^2}$ introduced in \cite{Guillera2003} and reproduced in Table \ref{table:1}. 
 
 If we use both the recurrences for $x=1, y=1$, we obtain a series for $\zeta(3) $ given in Lupa\c{s}'s paper \cite{Lupas2000}. Now, let us further simplify 
 our expression for the latter recursion given in Theorem \ref{experimental6F51}, letting $h$, $r_{1}$, and $r_{2}$ be as below: 
\begin{align*}
 & \text{summand}(x, y, k) = \frac{\left(x+\frac{y-1}{2}\right) 
 (x)_k^4 \left(x+\frac{y + 1}{2}\right)_k}{\left(x+\frac{y-1}{2}\right)_k 
 (x + y)_k^4}, \\
 & r_{1}(x, y) = \frac{6 (x-1) y+2 (x-1)^2+5 y^2}{4 (2 y-1)}, \\
 & r_{2}(x, y) = -\frac{y^5}{2 (2 y-1) (x+y)^4}. 
\end{align*}
 So, we obtain that 
\begin{equation}\label{hr1r2h}
 \sum _{k=0}^{\infty} \text{summand}(x,y,k) = r_{1}(x,y) + r_{2}(x,y) \sum _{k = 0}^{\infty} \text{summand}(x,y+1,k), 
\end{equation}
 leading us to the following. 
 A direct application of Lemma \ref{generalizeGuillera}, 
 for $\text{summand}(x, y, k)$, $r_{1}(x, y)$, and $r_{2}(x, y)$ as above, may 
 be regarded as a generalization of Guillera's series for $\frac{1}{\pi^2}$ introduced in \cite{Guillera2003} and reproduced in Table 
 \ref{table:1}, in view of our derivation of this formula along with the following results. 

\begin{example}
 Setting $x = \frac{3}{2}$ and $y = \frac{3}{2}$, our series acceleration formula given in Lemma \ref{generalizeGuillera} yields the following: 
 $$ \frac{128}{\pi ^2} = \sum _{j=0}^{\infty} \left(-2^{-12} \right)^j \binom{2 j}{j}^5 \frac{20j^2+32 j+13}{(j+1)^4}. $$ 
\end{example}

\begin{example}
 Setting $x = \frac{5}{2}$ and $y = \frac{3}{2}$, Lemma \ref{generalizeGuillera} gives us that: 
 $$ \frac{2048}{81\pi ^2} 
 = \sum _{j=0}^{\infty} \left(-2^{-12} \right)^j \binom{2 j}{j}^5 \frac{20j^2+56 j+41}{(j+1)^4(j+2)^4}. $$ 
\end{example}

\begin{example}
 Setting $x = \frac{1}{2}$ and $y = 1$, Lemma \ref{generalizeGuillera} yields 
 $$ 28 \zeta (3) = \sum _{j=1}^{\infty} \frac{ \left(-2^{8} \right)^j \left( - 10 j^2 +6 j - 1 \right)}{j^5 \binom{2 j}{j}^5}. $$
\end{example}

\begin{example}
 If we use both recurrences from Theorem \ref{experimental6F51} one after the other, with $x = \frac{3}{2}$ and $y = \frac{3}{2}$, 
 we obtain Guillera's series for $\frac{1}{\pi^2}$ from \cite{Guillera2008} reproduced in Table \ref{table:1}. It has convergence rate $\frac{1}{2^{10}}$.
\end{example}

\section{On the experimental discovery of hypergeometric transforms}\label{sectionnoncomp}
 In their classic text, \cite{BorweinBaileyGirgensohn2004} 
 describe how a human-computer dialogue is central to experimental mathematics 
 \cite[p.\ vii]{BorweinBaileyGirgensohn2004}. This is exemplified by the use of computers to explore new areas in mathematics, and via the development of a 
 better understanding of mathematical phenomena using computer-based results and computer-based tools \cite[p.\ vii--viii]{BorweinBaileyGirgensohn2004}. 
 In this regard, we have used a computer-based approach to study the problem of accelerating series, and our computer explorations have led us to identify 
 patterns that show us that many previously published results by many different authors, including  \cite{Ramanujan1914},  
 \cite{Guillera2003,Guillera2008},      
 \cite{HessamiPilehroodHessamiPilehrood2008,HessamiPilehroodHessamiPilehrood2010}, and  \cite{Lupas2000}, are special cases of our methods. 
 We find it appropriate to explain how we had obtained and discovered our main results in an experimental way. 

 In 2018 \cite{LevrieNimbran2018}, the recurrence $$ {}_{3}F_{2}\!\!\left[ \begin{matrix}x,x,1 \vspace{1mm}\\ x+y,x+y \end{matrix} \ \Bigg| \
 {1} \right] =\frac{2x+3y - 2}{2(2y-1)} + \frac{y^3}{2(2y-1)(x + y)^2}\cdot 
 {}_{3}F_{2}\!\!\left[ \begin{matrix}x+1,x+1,1 \vspace{1mm}\\ x+1+y,x+1+y \end{matrix} \ \Bigg| \
 {1} \right] $$ was introduced and proved. If we write 
\begin{equation}\label{LevrieNimbrans}
 s(x, y) = {}_{3}F_{2}\!\!\left[ \begin{matrix}x,x,1 \vspace{1mm}\\ x+y,x+y \end{matrix} \ \Bigg| \
 {1} \right] = 
 \sum_{n = 0}^\infty \frac{(x)_n^2}{(x+y)_n^2} 
\end{equation}
 and then write $f(x,y,n)=\frac{2 n +2 x -y}{2(2y -1)}$, we find that this summand may be manipulated in the below indicated manner, starting with the 
 reindexing argument shown below, and noting that a telescoping argument is implicitly used to simplify the below expression involving $f(x, y, n)$ 
 and $f(x, y, n+1)$: 
\begin{align*}
 s(x,y) & = 1 + \sum_{n=0}^\infty \frac{(x)_{n+1}^2}{(x+y)_{n+1}^2} \\ 
 & = 1 + \sum_{n=0}^\infty 
 \frac{(x)_{n}^2}{(x+y)_{n}^2}\cdot \frac{\left(x +n \right)^{2}}{\left(x +y +n \right)^{2}} \\
 & = 1+\sum_{n = 
 0}^\infty \frac{(x)_{n}^2}{(x+y)_{n}^2} \left(f(x, 
 y,n) - f(x,y,n+1)\frac{\left(x +n \right)^{2}}{\left(x +y +n \right)^{2}} + \frac{y^{3}}{2 \left(2 y -1\right) \left(x +y +n \right)^{2}} \right) \\
 & = 1+f(x,y,0)+ \frac{y^{3}}{2 \left(2 y -1\right)} \sum_{n = 
 0}^\infty \frac{(x)_{n}^2}{(x+y)_{n+1}^2 } \\
 & = \frac{2x+3y-2}{2(2y-1)}+ \frac{y^{3}}{2 \left(2 y -1\right)(x+y)^2} \sum_{n=0}^\infty \frac{(x)_{n}^2}{(x+y+1)_{n}^2}. 
\end{align*}
 Our approach toward obtaining and proving hypergeometric recurrences is similar relative to the above derivation and involves computer-based 
 experiments with summands as in \eqref{LevrieNimbrans}. Our research is based on explorations involving the use of Computer Algebra Systems to 
 experiment with hypergeometric series in order to emulate the telescoping argument involving $f(x, y, n)$, as above. This has led us toward patterns based on 
 our computational experiments that generalize closed forms due to Ramanujan, Lupa\c{s}, and Guillera, and others. This is motivated by open 
 problems concerning Ramanujan's and Guillera's series, such as the problem as to how it may be possible to generalize Guillera's series so as to obtain 
 hypergeometric expansions for $\frac{1}{\pi^k}$ for higher powers, with  
     of $k > 2$ \cite{Guillera2020}. As below, we offer some brief descriptions as to how we   
  had discovered the hypergeometric transforms given in this article. 

 For Theorem \ref{theoremfirst3F21}, a similar approach was used, relative to the above manipulations of the summand for \eqref{LevrieNimbrans}, 
 and the rational function identity given below was involved: 
\begin{align*}
 & \frac{(x+n)(x + n + 1)}{(x +y +n )(x+y+n+1)} = \\
 & \frac{2n +2x-y+1}{2(2y -1)}-\frac{2n+2 + 
 2x-y+1}{2(2y -1)} \frac{(x+n)(x+n+1)}{(x +y +n )(x+y+n+1)} \\ 
 & +\frac{(y-1)y(y+1)}{2 (2 y -1)}\frac{1}{(x +y +n )(x+y+n+1)}.
\end{align*}
 Similarly, for Theorem \ref{theoremfurther3F2n1}, summand manipulations were applied 
 to $$ s(x,y)=\sum_{n=0}^\infty (-1)^n \frac{(x)_n^2}{(y)_n^2} $$ 
 so as to obtain the identity 
\begin{align*} 
 & (x+n)^2 = \\
 & \text{summand}(x, y, n) + 
 \text{summand}(x, y, n + 1)\frac{(x+n)^2 }{(x+y+n+1)^2} + 
 \frac{y(1+y)^3}{4(x+y+n+1)^2} 
\end{align*}
 with $ \text{summand}(x,y,n) = \frac{2n^2+2(2x + y)n + 2x^2 + 2xy - y^2-y}{4}$. 
 Similar approaches led to our discoveries for the remaining Theorems in this article. 

\section{Conclusion}\label{sectionConclusion}
 In an article due to \cite{Guillera2020}, 
 conjectured Ramanujan-like series for $\frac{1}{\pi^k}$ for $k \geq 3$ were recorded, such as the formula 
\begin{align*}
 & \sum_{n=0}^{\infty} \frac{ \left( \frac{1}{2} \right)_{n}^{5} 
 \left( \frac{1}{3} \right)_{n} \left( \frac{2}{3} \right)_{n} 
 \left( \frac{1}{4} \right)_{n} \left( \frac{3}{4} \right)_{n} }{ \left( 1 \right)_{n}^{9} }
 (4528 n^4 + 3180 n^3 + 972 n^2 + 147 n + 9) 
 \left( -\frac{27}{256} \right)^{n} \\
 & = \frac{768}{\pi^4} 
\end{align*}
 conjectured by Zhao. Guillera also proved a bilateral series expansion for negative powers of $\pi$ \cite{Guillera2020}. This was achieved via the 
 periodicity of the function $$ f(x) = e^{-i \pi x} (\cos \pi x)^{2j+1} \prod_{s_{k} \neq \frac{1}{2}} \frac{\cos \pi x - \cos \pi s_{k}}{1 - \cos \pi s_{k}} 
 \sum_{n = -\infty}^{\infty} A(n, x) $$ for a function $A(n, x)$ of the form $$ A(n, x) = \left( \prod_{i=0}^{2m} \frac{ (s_{i})_{n + x} }{ (1)_{n+x} } 
 \right) \sum_{k=0}^{m} a_{k} (n+x)^{k} z_{0}^{n+x}. $$ We encourage the pursuit of research based on how our recursive and WZ-based methods may 
 be applied to obtain series for $\frac{1}{\pi^k}$ for higher   values  of $k > 2$. In this regard, the main obstacle is due to how difficult it is to obtain series for 
 $\frac{1}{\pi^{k > 2}}$ that may be accelerated according to our methods. For example, there are classically known series for $\frac{1}{\pi^2}$ that may 
 be easily proved inductively and that may be accelerated using our hypergeometric transforms, but it is not obvious as to how it may be possible to start 
 with a relatively ``simple'' series for $\frac{1}{\pi^3}$ or $\frac{1}{\pi^4}$ and then apply our evaluation technique. One possibility to consider might be 
 based on manipulations of the Cauchy product of known series for $\frac{1}{\pi^2}$, and we leave it to a separate project to investigate this option. 

 Consider the two recurrences for the hypergeometric expression $s(x, y)$ involved in Section \ref{GeneralizationsGuillera}. As above, we have applied these 
 recurrences one after the other. If we instead apply these recurrences so as to express $s(x, y)$ in terms of $s(x+3, y+3)$, then, by mimicking our 
 proof of Guillera's series from \cite{Guillera2008} shown in Table \ref{table:1}, we can show that 
\begin{equation}\label{fromseparate}
 \frac{8}{\pi ^2} 
 = \sum_{j=0}^{\infty} \left( -2^{-60} \right)^{j} \binom{6 j}{3 j}^5
 \frac{ p(j)}{(2 j-1)^5 (6 j-5)^3 (6 j-1)^5}, 
\end{equation} where $ p(j) $
 denotes the following polynomial.
\begin{verbatim}
-125 + 950*j + 31660*j^2 - 50024*j^3 - 14995984*j^4 + 
235669685116640*j^5 - 5284712547631680*j^6 + 52653694429319040*j^7 - 
306803615328753408*j^8 + 1157357544632934912*j^9 - 
2952738590842027008*j^10 + 5159445973725726720*j^11 - 
6097655344343592960*j^12 + 4666164796084690944*j^13 - 
2088682533545361408*j^14 + 415520397881671680*j^15
\end{verbatim}
 The above series for $\frac{1}{\pi^2}$ has a convergence rate of $\frac{1}{1073741824}$. Formulas as in \eqref{fromseparate} may also be determined 
 by applying series $n$-sections to the accelerated series given in this article, as explained below, and we encourage the investigation of this subject. 

 Let us begin, for example, with the series for $\pi^2$ from \cite{Campbell2022AMEN} recorded in Table \ref{table:1}. By applying a series bisection to 
 this series, we find that 
\begin{align*}
 & \frac{\pi ^2}{4} = 
 \sum _{n=1}^{\infty } \frac{2^{8 n} (2 n+1) (6 n+1)}{2 n (4 n+1)^2 \binom{4 n}{2 n}^3 } + 
 \sum _{n = 1}^{\infty } \frac{2^{8 n} n (3 n-1)}{4 (4 n-1)^2 (2 n-1) \binom{4 n-2}{2 n-1}^3}, 
\end{align*}
 and a reindexing-type argument then gives us that 
\begin{equation}\label{recentlydifferent}
 \pi^2 = \sum_{n=1}^{\infty} 2^{8 n} \frac{240 n^4-8 n^3-40 n^2-n+1}{n^2 (2 n-1) (4 n+1)^2 \binom{4 n}{2 n}^3}, 
\end{equation}
 and this was recently proved in a different way, via the WZ method, by \cite{Campbellpreprint}. We see that \eqref{recentlydifferent} has an 
 acceleration rate of $\frac{1}{16}$, relative to the acceleration rate of $\frac{1}{4}$ in the original formula from \cite{Campbell2022AMEN}. Chu and 
 Zhang used the same kind of acceleration approach via series bisections in Example 89 in \cite{ChuZhang2014}, and we may obtain similarly accelerated 
 series via series $n$-sections, as opposed to series bisections. 

 We have numerically found that $$ \frac{\pi ^4}{8} = \sum _{n=1}^{\infty} \frac{2^{8 n} (3 n-1) \left(7 n^2-5 n+1\right)}{n^7 \binom{2 n}{n}^7}. $$ 
 How can this conjectured formula be proved using our methods, or otherwise? 

\acknowledgements
\label{sec:ack}
 The authors want to thank a number of expert Reviewers who have offered many helpful suggestions that have considerably improved 
 the presentation of our work. 

\nocite{*}
\bibliographystyle{abbrvnat}
% use the following instead if you encounter problems 
%\bibliographystyle{alpha}
\bibliography{AccelerationV30}
\label{sec:biblio}

\end{document}